\documentclass[10pt]{article}

\usepackage{lmodern}
\usepackage[T1]{fontenc}
\usepackage{amsmath}
\usepackage{amsthm}
\usepackage{amssymb}
\usepackage{url}
\usepackage{latexsym}
\usepackage[small]{titlesec}
\usepackage[small,it]{caption}
\usepackage{mathtools} % for mathclap

\usepackage[square,comma,numbers,sort&compress]{natbib}

\usepackage{lineno}
\usepackage[bookmarks]{hyperref}
\hypersetup{
	colorlinks=true,
	linkcolor=black,
	anchorcolor=black,
	citecolor=black,
	urlcolor=black,
	pdfpagemode=UseThumbs,
	pdftitle={},
	pdfsubject={Combinatorics},
	pdfauthor={Jay Pantone},
}

\newcounter{todocounter}

\newcommand{\ds}{\displaystyle}
\newcommand{\vx}{\textbf{x}}
\newcommand{\vy}{\textbf{y}}

\newcommand{\vc}{\textbf{c}}

\newcommand{\VV}{\mathcal{V}}
\newcommand{\vD}{\textbf{D}}
\newcommand{\vT}{\textbf{T}}
\newcommand{\Z}{\mathbb{Z}}
\newcommand{\R}{\mathbb{R}}
\newcommand{\N}{\mathbb{N}}
\newcommand{\C}{\mathbb{C}}
\newcommand{\vone}{\textbf{1}}

\theoremstyle{plain}
\newtheorem{theorem}{Theorem}[section]
\newtheorem{proposition}[theorem]{Proposition}

\theoremstyle{definition}

\makeatletter
\newfont{\footsc}{cmcsc10 at 8truept}
\newfont{\footbf}{cmbx10 at 8truept}
\newfont{\footrm}{cmr10 at 10truept}
\renewenvironment{abstract}{
	\begin{list}{}%
	{\setlength{\rightmargin}{1in}%
	\setlength{\leftmargin}{1in}}%
	\item[]\ignorespaces\begin{small}}%
	{\end{small}\unskip\end{list}%
}

\newcommand*\patchAmsMathEnvironmentForLineno[1]{%
  \expandafter\let\csname old#1\expandafter\endcsname\csname #1\endcsname
  \expandafter\let\csname oldend#1\expandafter\endcsname\csname end#1\endcsname
  \renewenvironment{#1}%
     {\linenomath\csname old#1\endcsname}%
     {\csname oldend#1\endcsname\endlinenomath}}% 
\newcommand*\patchBothAmsMathEnvironmentsForLineno[1]{%
  \patchAmsMathEnvironmentForLineno{#1}%
  \patchAmsMathEnvironmentForLineno{#1*}}%
\AtBeginDocument{%
\patchBothAmsMathEnvironmentsForLineno{equation}%
\patchBothAmsMathEnvironmentsForLineno{align}%
\patchBothAmsMathEnvironmentsForLineno{flalign}%
\patchBothAmsMathEnvironmentsForLineno{alignat}%
\patchBothAmsMathEnvironmentsForLineno{gather}%
\patchBothAmsMathEnvironmentsForLineno{multline}%
}
\newpagestyle{main}[\small]{
	\headrule
	\sethead[\usepage][][]
	{\sc The Asymptotic Number of Singular Vector Tuples}{}{\usepage}
}

\title{\sc The Asymptotic Number of Simple Singular Vector Tuples of a Cubical Tensor}

\author{%
	Jay Pantone\\
	\small Department of Mathematics\\[-3pt]
	\small Dartmouth College\\[-3pt]
	\small Hanover, New Hampshire USA
}

\titleformat{\section}{\large\sc}{\thesection.}{1em}{}
\date{}
 
\begin{document}
\maketitle

\pagestyle{main}

\begin{abstract}
	S. Ekhad and D. Zeilberger recently proved that the multivariate generating function for the number of simple singular vector tuples of a generic $m_1 \times \cdots \times m_d$ tensor has an elegant rational form involving elementary symmetric functions, and provided a partial conjecture for the asymptotic behavior of the cubical case $m_1 = \cdots = m_d$. We prove this conjecture and further identify completely the dominant asymptotic term, including the multiplicative constant. Finally, we use the method of differential approximants to conjecture that the \emph{subdominant connective constant} effect observed by Ekhad and Zeilberger for a particular case in fact occurs more generally.
\end{abstract}

\section{Introduction}

In this note, we confirm a conjecture of Ekhad and Zeilberger~\cite{zeilberger:singular-vector-tuples} regarding the number of simple singular vector tuples of a generic $m_1 \times \cdots \times m_d$ complex tensor. We refer the reader to the work of Friedland and Ottaviani~\cite{friedland:singular-vector-tuples} for the definitions of these terms, as they are not important to the content of this article. Therein, the authors prove the following theorem.

\begin{theorem}[{Friedland and Ottaviani~\cite[Theorem 1]{friedland:singular-vector-tuples}}]
	\label{theorem:fried}
	The number of simple singular vector tuples of a generic $m_1 \times \cdots \times m_d$ complex tensor is equal to the coefficient of $t_1^{m_1-1}\cdots t_d^{m_d-1}$ in the expression
	\[
		\prod_{i=1}^d \frac{\widehat{t_i}^{m_i} - {t_i}^{m_i}}{\widehat{t_i}-t_i}, \qquad \text{where \;\; } \widehat{t_i} = \left(\sum_{j=1}^d t_j\right) - t_i.
	\]
\end{theorem}

Denoting by $a_d(m_1, \ldots, m_d)$ the quantity described in Theorem~\ref{theorem:fried}, Ekhad and Zeilberger~\cite{zeilberger:singular-vector-tuples} derived a rational generating function for this multi-indexed sequence.

\begin{theorem}[{Ekhad and Zeilberger~\cite[Proposition 1]{zeilberger:singular-vector-tuples}}]
	\label{theorem:zeil}
	Let $e_i(x_1, \ldots, x_d)$ be the $i$th elementary symmetric function
	\[
		e_i(x_1, \ldots, x_d) = \ds\sum_{1 \leq r_1 < \cdots < r_i \leq d} x_{r_1}\cdots x_{r_i}.
	\]
	Then, the multivariate generating function $A_d(x_1, \ldots, x_d)$ for the sequence $a_d(m_1, \ldots, m_d)$ is
	\[
		A_d(x_1, \ldots, x_d) = \sum_{\mathclap{m_1,\ldots, m_d \geq 0}} a_d(m_1, \ldots, m_d) x_1^{m_1} \ldots x_d^{m_d} = \left(1 - \ds\sum_{i=2}^d (i-1)e_i(x_1, \ldots, x_d)\right)^{-1}\prod_{i=1}^d \frac{x_i}{1-x_i}.
	\]
\end{theorem}

We are primarily interested in the number of simple singular vector tuples of tensors for which $m_1 = \cdots = m_d$, known as \emph{cubical tensors}. Denote
\[
	C_d(n) = a_d(\underbrace{n, \ldots, n}_{d\text{ times}}),
\]
and observe that Theorem~\ref{theorem:zeil} implies that the generating function $F_d(x)$ of the sequence $\{C_d(n)\}_{n \geq 0}$ is the diagonal of $A_d(x_1, \ldots, x_d)$; that is,
\[
	F_d(x) = \ds\sum_{n \geq 0} C_d(n)x^n = \ds\sum_{n \geq 0} [x_1^n \cdots x_d^n] A_d(x_1, \ldots, x_d)x^n.
\]
Here $[x_1^n\cdots x_d^n]A(x)$ denotes the coefficient of $x_1^n\cdots x_d^n$ in $A(x)$.

A univariate generating function $A(x)$ is said to be D-finite if it is the solution of a nontrivial linear differential equation with polynomial coefficients (in $x$), and a sequence $a(n)$ is said to be P-recursive if it satisfies a recurrence relation of the form
\[
	p_0(n)a(n) + p_1(n)a(n-1) + \cdots + p_k(n)a(n-k) = 0
\]
where each $p_i(n)$ is a polynomial and $p_0(n) \neq 0$. These two notions are in fact equivalent---a generating function $A(x)$ is D-finite if and only if the coefficients of its power series expansion are P-recursive.

The theory of D-finite functions (see, e.g., Zeilberger~\cite{zeilberger:sister-celine}, Christol~\cite{christol:diagonales}, and Lipshitz~\cite{lipshitz:diagonal-D-finite}) guarantees that each of the functions $F_d(x)$ is D-finite, as they are diagonals of rational functions. Unfortunately, current implementations of constructive approaches to finding $F_d(x)$ cannot handle even $d=5$.

Ekhad and Zeilberger~\cite{zeilberger:singular-vector-tuples} provide the recurrence relation for $C_3(n)$ and use this to find that the asymptotic behavior of the sequence is
\[
	C_3(n) \sim \frac{2}{\pi\sqrt{3}} \,8^n\,n^{-1}.
\]
The exponential growth rate $8$ (sometimes called the \emph{connective constant}) and the polynomial exponent $-1$ are derived rigorously from the recurrence relation for $C_3(n)$, while the multiplicative constant $2/(\pi\sqrt{3})$ is estimated through the calculation of many initial terms. After calculating $160$ initial terms of $C_4(n)$, Ekhad and Zeilberger further conjecture that 
\[
	C_4(n) \sim \alpha_4 \, 81^n \, n^{-3/2},
\]
for an unknown constant $\alpha_4$. Combining these with other numerical calculations, Ekhad and Zeilberger ultimately conjecture that
\[
	C_d(n) \sim \alpha_d \, ((d-1)^d)^n\,n^{(1-d)/2}.
\]

In Section~\ref{section:main} we confirm this conjecture, and more, by using multivariate asymptotic methods to prove the following theorem.
\begin{theorem}
	\label{theorem:asymp}
	For $d \geq 3$, the number $C_d(n)$ of simple singular vector tuples of a $d$-dimensional $n \times \cdots \times n$ cubical tensor is asymptotically
	\[
		C_d(n) = \frac{(d-1)^{d-1}}{(2\pi)^{(d-1)/2}d^{(d-2)/2}(d-2)^{(3d-1)/2}} \, ((d-1)^d)^n \, n^{(1-d)/2} \left( 1 + O\left(\frac{1}{n}\right)\right).
	\]
\end{theorem}

In Section~\ref{section:computations} we discuss the intractability of the computational problem of determining $F_d(x)$ exactly for small $d$, and we apply the method of differential approximants to explore the phenomenon of subdominant connective constants.

\section{The Asymptotic Behavior of $C_d(n)$}
\label{section:main}

Only recently have the techniques of analytic combinatorics been reliably extended to the multivariate case. In this section we appeal primarily to two articles of Raichev and Wilson~\cite{raichev:asymptotics-improvements-smooth, raichev:new-method-asymptotics} and one of Pemantle and Wilson~\cite{pemantle:twenty-combinatorial-examples}. We start by repeating the necessary definitions and theorems from these articles.

For a $d$-dimensional complex vector $\vx$, define
\begin{align*}
	G_d(\vx) &= \prod_{i=1}^d x_i,\\
	H_d(\vx) &= \left(\prod_{i=1}^d (1-x_i)\right) \left(1 - \ds\sum_{i=2}^d (i-1)e_i(\vx)\right),
\end{align*}
so that $A_d(\vx) = G_d(\vx)/H_d(\vx)$ is the generating function whose main diagonal asymptotic behavior we wish to compute. Going forward, we will drop the subscript when the context is clear.

Let $\VV$ be the variety defined by $H(\vx) = 0$. For complex $\vx$, define the polydisk $\vD(\vx)$ and the torus $\vT(\vx)$ by 
\begin{align*}
	\vD(\vx) &= \left\{\vx' : |x_i'| \leq |x_i| \text{ for all } i\right\},\\
	\vT(\vx) &= \left\{\vx' : |x_i'| = |x_i| \text{ for all } i\right\}.
\end{align*}

A point $\vx \in \VV$ is said to be \emph{minimal} if all of its coordinates are non-zero and $\VV \cap \vD(\vx) \subset \vT(\vx)$. Further, $\vx$ is \emph{strictly minimal} if it is the only point of $\VV$ in $\vT(\vx)$. 

For many practical examples, the primary obstacle in computing the asymptotic expansion of the diagonal (or more generally, the asymptotic expansion in any direction) is detecting which points of $\VV$ contribute to the asymptotic behavior. We will use a variety of direct calculations to show that for each $d$, the asymptotic behavior of the sequence $C_d(n)$ is governed by a single strictly minimal point in the positive orthant $\R_+^d$.

Theorem~\ref{theorem:asymp} will then be proved by applying a theorem of Raichev and Wilson~\cite{raichev:asymptotics-improvements-smooth}. The theorem is stated below; we have simplified it to apply only to asymptotic behavior along the main diagonal $F_{n\vone}$ of a $d$-variate generating function $F(\vx) = G(\vx)/H(\vx)$.\footnote{More specifically, we have substituted $\alpha = \vone$ and $p=1$.} Definitions of new terms are given after the statement of the theorem, and we use the short-hand $\partial_i H(\vy)$ to denote the partial derivative of $H(\vx)$ with respect to $x_i$, evaluated at $\vx = \vy$. We also use $\widehat{\vx_i}$ to denote $\vx$ with the $i$th coordinate deleted and $\widehat{\vx_{i,j}}$ to denote $\vx$ with the $i$th and $j$th coordinates deleted.

\begin{theorem}[{Raichev and Wilson~\cite[Theorem 3.2]{raichev:asymptotics-improvements-smooth}}]
	\label{theorem:rw-compute}
	Let $d \geq 2$. If $\vc \in \VV$ is strictly minimal, smooth with $c_d\partial_d H(\vc) \neq 0$, critical and isolated, and nondegenerate, then for all $N \in \N$,
	\[
		F_{n\vone} = \vc^{-n\vone}\left[\left(\left(2\pi n\right)^{d-1}\det \widetilde{g}''(0)\right)^{-1/2} \sum_{k < N} n^{-k}L_k(\widetilde{u}_0,\widetilde{g}) + O\left(n^{-(d-1)/2-N}\right)\right]
	\]
	as $n \to \infty$.
\end{theorem}

The quantities $\det \widetilde{g}''(0)$, $L_k$, $\widetilde{u}_0$, and $\widetilde{g}$ will be defined later, as needed. For now it suffices to remark that they can all be computed and hence Theorem~\ref{theorem:rw-compute} permits the computation of the asymptotic behavior of the main diagonal to arbitrary precision. 

There are a number of hypotheses that must be verified to apply Theorem~\ref{theorem:rw-compute}. We have already stated what it means for a point in $\VV$ to be strictly minimal. Further, $\vc \in \VV$ is \emph{smooth} if $\partial_i H(\vc) \neq 0$ for some $i$, $\vc \in \VV$ is \emph{critical} if it's smooth and 
\[
	c_1\partial_1H(\vc) = c_2\partial_2H(\vc) = \cdots = c_d\partial_d H(\vc),
\]
$\vc$ is \emph{isolated} if there is a neighborhood around $\vc$ in which it is the only critical point, and $\vc$ is \emph{nondegenerate} if $\det \widetilde{g}''(0) \neq 0$.

We claim that
\[
	\vc = \bigg(\underbrace{\frac{1}{d-1}, \ldots, \frac{1}{d-1}}_{d\text{ times}}\bigg),
\]
satisfies the hypotheses of Theorem~\ref{theorem:rw-compute} and therefore is the sole contributing point to the asymptotic behavior of $C_d(n)$. The verification of this claim relies on tedious computation using several properties of the symmetric functions $e_i(\vx)$; these will be stated as they are required. To simplify notation, denote
\begin{align*}
	P(\vx) &= \prod_{i=1}^d (1-x_i),\\
	S(\vx) &= 1 - \ds\sum_{i=2}^d (i-1)e_i(\vx).
\end{align*}

\begin{proposition}
	The point $\vc$ lies in the variety $\VV$.	
\end{proposition}
\begin{proof}
	It suffices to show that $S(\vc) = 0$. Observe that
	\[
		e_i(k\vone) = {d \choose i}k^i,
	\]
	and therefore
	\[
		S(\vc) = 1 - \ds\sum_{i=2}^d(i-1)e_i(\vc) = 1 - \ds\sum_{i=2}^d (i-1){d \choose i}\left(\frac{1}{d-1}\right)^i = 0,
	\]
	the final equality being verified by the computer algebra system Maple (which itself employs an algorithm of Zeilberger~\cite{zeilberger:holonomic-systems}).
\end{proof}

\begin{proposition}
	The point $\vc$ is strictly minimal in $\VV$.
\end{proposition}
\begin{proof}
	The variety $\VV$ can be written as the union
	\[
		\VV = \left\{ \vx : x_i = 1\text{ for some $i$} \right\} \; \cup \; \left\{ \vx : S(\vx) = 0 \right\}.
	\]
	This union is not disjoint.
	
	Suppose that $\vc$ were not minimal. Then, there would exist a minimal point $\vy \in \VV \cap D(\vc)$ different from $\vc$. Since $|y_i| \leq 1/(d-1)$ for all $i$, we must have $S(\vy) = 0$.
	
	 Consider the variety $\VV'$ defined by $S(\vy) = 0$. We say that a polynomial $P$ is \emph{aperiodic} if the set of integer combinations of the exponent vectors of its monomials is all of $\Z^d$. For example, $x_1 + x_1^2x_2$ is aperiodic because the the $\Z$-span of $\{(1, 0), (2, 1)\}$ is $\Z^2$, while $x_1^2 + x_2^2$ is not aperiodic. 	
	
	 Proposition 3.17 from Pemantle and Wilson~\cite{pemantle:twenty-combinatorial-examples} states that if $H = 1 - P$ for an aperiodic polynomial $P$, then every minimal point of the variety defined by $H=0$ is strictly minimal and lies in the positive orthant.\footnote{The proof of Proposition 3.17 in~\cite{pemantle:twenty-combinatorial-examples} does not rely on their Assumption 3.6.} Applying this to $S$ and $\VV'$, we conclude that $0 < y_i \leq 1/(d-1)$ for all $i$.
	 
	 It follows that
	 \[
	 	\ds\sum_{i=2}^d (i-1)e_i(\vy) \leq \sum_{i=2}^d (i-1){d \choose i}\left(\frac{1}{d-1}\right)^i = 1,
	 \]
	 with equality only when $\vy = \vc$. Therefore $\vc$ is minimal, and again by Proposition 3.17 from~\cite{pemantle:twenty-combinatorial-examples}, $\vc$ is in fact strictly minimal.
\end{proof}

\begin{proposition}
	\label{prop:c-smooth}
	The point $\vc$ is smooth.
\end{proposition}
\begin{proof}
	Observe first that
	\[
		\frac{\partial}{\partial x_j} e_i(\vx) = e_{i-1}(\widehat{\vx_j}).
	\]
	Therefore,
	\begin{align*}
		\partial_d H(\vc) &= (\partial_d P(\vc))S(\vc) + P(\vc)(\partial_d S(\vc))\\
		&= -\left(\prod_{i=1}^{d-1} (1-c_i)\right)S(\vc) + P(\vc)\left(-\ds\sum_{i=2}^d(i-1)e_{i-1}(\widehat{\vc_d})\right)\\
		&= -\left(\frac{d-2}{d-1}\right)^{d}\left(\frac{d}{d-1}\right)^{d-2}\\
		&= -\frac{(d-2)^{d}d^{d-2}}{(d-1)^{2d-2}} \neq 0.
	\end{align*}
	Thus $\vc$ is a smooth point.
\end{proof}

\begin{proposition}
	The point $\vc$ is critical. 
\end{proposition}
\begin{proof}
	To prove the criticality of $\vc$ we must verify that $c_j \partial_j H(\vc) = c_k\partial_kH(\vc)$ for all $j,k$. As both $\vc$ and $H$ are symmetric, this is trivially true.
\end{proof}

\begin{proposition}
	The point $\vc$ is isolated.
\end{proposition}
\begin{proof}
	Let $\epsilon > 0$ be small and let $B$ be the $\epsilon$-neighborhood around $\vc$. Suppose $\vy \in B$ is critical. It must then be true that
	\[
		y_1 \partial_1 H(\vy) = y_d \partial_d H(\vy).
	\]
	Using the calculation performed in Proposition~\ref{prop:c-smooth} along with the fact that $S(\vy) = 0$, it follows that
	\[
		y_1 \sum_{i=2}^d (i-1)e_{i-1}(\widehat{\vy_1}) = y_d \sum_{i=2}^d (i-1)e_{i-1}(\widehat{\vy_d}).
	\]
	
	Using the identity
	\[
		e_i(x_1, \ldots, x_d) = x_1e_{i-1}(x_2, \ldots, x_d) + e_i(x_2, \ldots, x_d)
	\]
	(with the convention that $e_d(x_2, \ldots, x_d) = 0$), the equality becomes
	\[
		y_1 \sum_{i=2}^d (i-1)\left(y_de_{i-1}(\widehat{\vy_{1,d}}) + e_i(\widehat{\vy_{1,d}})\right) = y_d \sum_{i=2}^d (i-1)\left(y_1e_{i-1}(\widehat{\vy_{1,d}}) + e_i(\widehat{\vy_{1,d}})\right).
	\]
	By canceling like terms, we see
	\[
		y_1 \sum_{i=2}^d (i-1)e_i(\widehat{\vy_{1,d}}) = y_d \sum_{i=2}^d (i-1)e_i(\widehat{\vy_{1,d}}).
	\]
	Since the function
	\[
		U(\vx) = \ds\sum_{i=2}^d (i-1)e_i(\widehat{\vx_{1,d}})
	\]
	is nonzero and continuous at $\vx = \vc$, $\epsilon$ can be chosen small enough to ensure that $U(\vy) \neq 0$. Dividing both sides by $U(\vy)$ yields
	\[
		y_1 = y_d,
	\]
	and symmetry implies that $\vy$ has the form $y\vone$ for some $y \in \C$. 
	
	Noting that 
	\[
		H(y\vone) = 1 - \ds\sum_{i=2}^d (i-1){d \choose i}y^i = (y+1)^{d-1}((1-d)y+1),
	\]
	we find that $\vy = \vc$. Therefore, $\vc$ is isolated.
\end{proof}

The last hypothesis to check is that $\vc$ is nondegenerate. This amounts to checking that $\det \widetilde{g}''(0) \neq 0$. In non-symmetric cases, the definition of $\widetilde{g}''(0)$ is quite cumbersome. Thankfully, Proposition 4.2 of~\cite{raichev:asymptotics-improvements-smooth} proves that in cases where $H$ and $\vc$ are symmetric,
\[
	\det \widetilde{g}''(0) = dq^{d-1},
\]
where
\[
	q = 1 + \frac{c_1}{\partial_d H(\vc)}\left(\partial_{dd}H(\vc) - \partial_{1d}H(\vc)\right).
\]

\begin{proposition}
	The point $\vc$ is nondegenerate.
\end{proposition}
\begin{proof}
	We showed in Proposition~\ref{prop:c-smooth} that
	\[
		\partial_d H(\vx) = -\left(\prod_{i=1}^{d-1} (1-x_i)\right)S(\vx) + P(\vx)\left(-\ds\sum_{i=2}^d(i-1)e_{i-1}(\widehat{\vx_d})\right)\\
	\]
	and
	\[
		\partial_d H(\vc) = -\frac{(d-2)^{d}d^{d-2}}{(d-1)^{2d-2}}.
	\]
	Additionally, noting that the first term in the left summand and the second term in the right summand are independent of $x_d$, we have
	\begin{align*}
		\partial_{dd} H(\vx) &= -\left(\prod_{i=1}^{d-1} (1-x_i)\right) \partial_d S(\vx) + (\partial_d P(\vx))\left(-\ds\sum_{i=2}^d(i-1)e_{i-1}(\widehat{\vx_d})\right)\\
		&= 2\left(\prod_{i=1}^{d-1} (1-x_i)\right)\left(\ds\sum_{i=2}^d(i-1)e_{i-1}(\widehat{\vx_d})\right),
	\end{align*}
	proving
	\[
		\partial_{dd} H(\vc) = 2\left(\frac{d-2}{d-1}\right)^{d-1}\left(\frac{d}{d-1}\right)^{d-2} = \frac{2d^{d-2}(d-2)^{d-1}}{(d-1)^{2d-3}}.
	\]
	Furthermore,
	\begin{align*}
		\partial_{1d} H(\vx) &= (\partial_{1d} P(\vx))S(\vx) + (\partial_d P(\vx))(\partial_1 S(\vx)) + (\partial_1 P(\vx))(\partial_d S(\vx)) + P(\vx)(\partial_{1d} S(\vx))\\
		&= \left(\prod_{i=2}^{d-1} (1-x_i)\right)\left(S(\vx) + (1-x_1)\left(\sum_{i=2}^d(i-1)e_{i-1}(\widehat{\vx_1})\right)\right.\\
		& \qquad\qquad\qquad\qquad\qquad\qquad + (1-x_d)\left(\sum_{i=2}^d(i-1)e_{i-1}(\widehat{\vx_d})\right)\\
		& \qquad\qquad\qquad\qquad\qquad\qquad \left.- (1-x_1)(1-x_d)\left(\sum_{i=2}^d(i-1)e_{i-2}(\widehat{\vx_{1,d}})\right)\right),
	\end{align*}
	proving,
	\[
		\partial_{1d} H(\vc) = \left(\frac{d-2}{d-1}\right)^{d-1}\left(2\left(\frac{d}{d-1}\right)^{d-2} - 2\left(\frac{d-2}{d-1}\right)\left(\frac{d}{d-1}\right)^{d-3}\right) = \frac{4d^{d-3}(d-2)^{d-1}}{(d-1)^{2d-3}}.
	\]
	
	We can now compute $q$:
	\[
		q = 1 + \frac{c_1}{\partial_d H(\vc)}\left(\partial_{dd}H(\vc) - \partial_{1d}H(\vc)\right) = \frac{d-2}{d}.
	\]
	Finally,
	\[
		\det \widetilde{g}''(0) = dq^{d-1} = d\left(\frac{d-2}{d}\right)^{d-1} = \frac{(d-2)^{d-1}}{d^{d-2}} \neq 0.
	\]
	Hence, $\vc$ is nondegenerate.
\end{proof}

Having verified the hypotheses of Theorem~\ref{theorem:rw-compute} for $\vc$, we will now define and compute several of the quantities in its conclusion. To find the first-order asymptotic behavior, we consider the case $N=1$ in the theorem. Applying the appropriate simplifications to the definitions of $L_k$, $\widetilde{u}_0$, and $\widetilde{g}$ in~\cite{raichev:asymptotics-improvements-smooth}, we find that
	\[
		L_0(\widetilde{u}_0, \widetilde{g}) = \frac{G(\vc)}{-c_d\partial_dH(\vc)} = \left(\frac{1}{d-1}\right)^d\left(\frac{(d-1)^{2d-1}}{d^{d-2}(d-2)^{d}}\right) = \frac{(d-1)^{d-1}}{d^{d-2}(d-2)^{d}}.
	\]
	
Assembling all computed quantities into the conclusion of Theorem~\ref{theorem:rw-compute} yields
\[
	C_d(n) = \frac{L_0(\widetilde{u}_0, \widetilde{g})}{\sqrt{(2\pi)^{d-1} \det \widetilde{g}''(0)}} ((d-1)^d)^n n^{(1-d)/2}\left( 1 + O\left(\frac{1}{n}\right)\right)
\]
and so
\[
	C_d(n) = \frac{(d-1)^{d-1}}{(2\pi)^{(d-1)/2}d^{(d-2)/2}(d-2)^{(3d-1)/2}} ((d-1)^d)^n n^{(1-d)/2}\left( 1 + O\left(\frac{1}{n}\right)\right),
\]
proving Theorem~\ref{theorem:asymp}.  

The computation of the asymptotic behavior of off-diagonal sequences can also be performed using the same techniques. In this case, however, the loss of symmetry will complicate some of the necessary calculations.

\section{Computational Aspects and Subdominant Connective Constants}
\label{section:computations}

All known automatic methods for computing diagonals of rational functions, either exactly or asymptotically, suffer from large run-times. Recent advances have improved the situation, though such calculations still remain out of reach for even reasonably sized rational functions in more than a few variables. We comment on two such implementations.

Apagodu and Zeilberger~\cite{apagodu:smaz} provide an algorithm that produces a linear recurrence with polynomial coefficients (in $n$) for the diagonal coefficients of a rational function. Applying the algorithm to $C_3(n)$ returns, after a few hours, a recurrence of order $6$ with polynomial coefficients of degree at most $7$. We did not attempt to apply the algorithm to $C_4(n)$. Ekhad and Zeilberger~\cite{zeilberger:singular-vector-tuples} note that it is much faster to generate terms of the sequence $C_3(n)$ and guess a linear recurrence. More recently, Lairez~\cite{lairez:computing-periods} has provided a Magma implementation to produce the differential equation satisfied by the diagonal of a rational function. It finds the generating function for $C_3(n)$ in a few seconds and the generating function for $C_4(n)$ in about $40$ minutes. 

On the asymptotic side, recent work of Melczer and Salvy~\cite{melczer:symbolic-numeric-tools} provides an improved algorithm to rigorously compute the asymptotic behavior of diagonals of rational functions. Their implementation provides the correct asymptotic behavior for $C_3(n)$ and $C_4(n)$ in a few seconds, and that of $C_5(n)$ in a few minutes.

Upon the calculation of the linear recurrence satisfied by $C_3(n)$, Ekhad and Zeilberger note that for the correct initial conditions the connective constant (better known in some circles as the exponential growth rate) is $8$. However, for most other initial conditions, the resulting sequence would have connective constant $9$. Though we are not able to find linear recurrences for $C_d(n)$ for $d \geq 5$, we can provide some evidence that this phenomenon of subdominant connective constants persists for all values of $d$. 

We employ the method of differential approximants, pioneered by Guttmann and Joyce~\cite{guttmann:a-new-method} and a favorite tool of statistical mechanists. It allows for experimental estimation of the asymptotic behavior of a sequence given only a finite number of known initial terms. A forthcoming article~\cite{pantone:diff-approx-2} by the present author will explore the inner workings of the method, its usefulness to enumerative combinatorics, and provide an open-source implementation. 

Using the first $100$ terms of $C_3(n)$, the method of differential approximants predicts that the generating function $F_3(x)$ has, as expected, a singularity located at
\[
	x \approx 0.12500000000000000000000000000001 \pm (2\cdot 10^{-32})
\]
corresponding to the known connective constant $8$. More interestingly, it also detects a singularity located at
\[
	x \approx 0.11111111111113 \pm (4 \cdot 10^{-14}).
\]
In most cases, this would imply a connective constant $9$. Being in that case the dominant singularity, we would expect it to be estimated \emph{more} accurately than the singularities further from the origin, not less. In our experience, this indicates that a sequence has a subdominant connective constant, as is known to be true in this case.

Applying the same process to the first $100$ terms of $C_4(n)$ yields estimates for the location of singularities of $F_4(x)$ at
\begin{align*}
	x &\approx 0.0123456790123456790123456790123456790123456790123457 \pm 2\cdot 10^{-52}\text{, and}\\
	x &\approx 0.00799999999999999 \pm (5 \cdot 10^{-17}).
\end{align*}
The first indicates the known connective constant $81$, while the second indicates that this connective constant is subdominant to a connective constant $125$.

The first $70$ terms of $C_5(n)$ are sufficient to predict the location of singularities of $F_5(x)$ to be
\begin{align*}
	x &\approx 0.000976562499999999999999999999999999996 \pm 4\cdot 10^{-39}\text{, and}\\
	x &\approx 0.0004164930 \pm (2 \cdot 10^{-10}),
\end{align*}
implying that the known connective constant $1024$ is subdominant to a connective constant $2401$.

This evidence leads us to conjecture that the known connective constants $(d-1)^d$ of all $C_d(n)$ are subdominant to the connective constants $(2d-3)^{d-1}$ for generic solutions to the linear recurrence for $C_d(n)$. \ \\

\textbf{Acknowledgements:} The author would like to thank the referees for their careful reading and feedback, which significantly improved this article.

\bibliographystyle{acm}
\bibliography{/Users/jay/Dropbox/Research/refs/bib.bib}

\end{document}